\documentclass[a4paper, 11pt, reqno]{amsart}
\usepackage[english]{babel}
\usepackage[utf8]{inputenc}
\usepackage[T1]{fontenc}
\usepackage{amsmath}
\usepackage{amssymb}
\usepackage{mathtools} 
\usepackage{color}
\usepackage{nicefrac}
\DeclareMathOperator{\sgn}{sgn}
\usepackage{a4wide}
\usepackage{amsrefs}
\usepackage{amsthm}
\usepackage{mathrsfs}
\usepackage{comment}
\usepackage{hyperref}

\theoremstyle{plain}
\newtheorem{thm}{Theorem}[section]
\newtheorem{prop}[thm]{Proposition}
\newtheorem{lemma}[thm]{Lemma}
\newtheorem{cor}[thm]{Corollary}

\theoremstyle{definition}
\newtheorem{defn}[thm]{Definition}

\theoremstyle{remark}
\newtheorem*{rmk}{Remark}
\newtheorem*{rmks}{Remarks}

\newcommand{\C}{\mathbb{C}}
\renewcommand{\H}{\mathbb{H}}
\newcommand{\Z}{\mathbb{Z}}
\newcommand{\Q}{\mathbb{Q}}
\newcommand{\N}{\mathbb{N}}
\newcommand{\R}{\mathbb{R}}

\newcommand{\im}{\textnormal{Im}}
\newcommand{\vt}[1]{\left\lvert #1 \right\rvert}

\makeatletter
\let\@@pmod\pmod
\DeclareRobustCommand{\pmod}{\@ifstar\@pmods\@@pmod}
\def\@pmods#1{\mkern4mu({\operator@font mod}\mkern 6mu#1)}
\makeatother

\binoppenalty=\maxdimen
\relpenalty=\maxdimen

\numberwithin{equation}{section}

\title{Modular Forms Related to Real Quadratic Fields as Traces of Cycle Integrals}
\author{Johann Stumpenhusen}
\address{Department of Mathematics and Computer Science, Division of Mathematics, University of Cologne, Weyertal 86-90, 50931 Cologne, Germany}
\email{jstumpen@math.uni-koeln.de}
\date{\today}

\begin{document}

\begin{abstract}
A decade ago, locally harmonic Maa{\ss} forms were first established by Bringmann, Kane, and Kohnen in negative weights and independently by Hövel in weight $0$. Since then, they have seen important applications related to twisted central $L$-values of newforms. However, there are just a few examples of these forms known so far, and in particular there is no unifying theory behind these examples. Towards this direction, we show that the representations of Zagier's $f_{k,D}$ and Mono's $g_{k+1,D}$ functions as traces of cycle integrals are not unique.
\end{abstract}



\maketitle

\section{Introduction and statement of results}

\subsection{Modular forms related to quadratic fields} We stipulate $\tau = u + iv$ and $z = x + iy$ with $u, x \in \R$ and $v, y \in \R_+$. Throughout the paper, if not stated otherwise, let $D \in \N$ be a positive non-square discriminant, $k \in \N$ even, $\mathcal{Q}_D$ the set of all binary integral quadratic forms of discriminant $D$ and $\H := \{\tau \in \C: v > 0\}$ the complex upper half-plane. In 1975, Zagier \cite{zagier75} introduced the functions
\[f_{k,D}: \H \to \C, \quad \tau \mapsto \sum_{Q \in \mathcal{Q}_D}\frac{1}{Q(\tau,1)^k}\]
which are cusp forms of weight $2k$ for $\Gamma := \operatorname{SL}_2(\Z)$ if $k > 2$ and, in fact, span $S_{2k}(\Gamma)$ \cite[Theorem 1]{katok}. Variations thereof have been investigated by Bengoechea \cite{ben13} (meromorphic cusp forms for $D < 0$) and Parson \cite{par93} (entire modular integrals for a partial sum and odd $k$).

In \cite{mo21}, Mono introduced the slightly modified functions
\[g_{k+1,D}: \H \to \C, \quad \tau \mapsto \sum_{Q \in \mathcal{Q}} \frac{\sgn(Q_\tau)}{Q(\tau,1)^{k+1}}\]
and showed that these are locally harmonic\footnote{In fact, they are even locally holomorphic.} Maa{\ss} forms (see Subsection \ref{subsec:LocallyHarmonicMaass} for a definition) with exceptional set $E_D  := \{\tau \in \H: Q_\tau = 0 \textup{ for some } Q \in \mathcal{Q}_D\}$ where $Q_\tau := \frac{a|\tau|^2 + bu + c}{v}$ for $Q(X,Y) = [a,b,c](X,Y) = aX^2 + bXY + cY^2$. This type of modular objects first appeared in work by Bringmann, Kane, and Kohnen \cite{bkk} as a result of constructing a locally defined modular object
\[\mathcal{F}_{1-k,D}: \H \to \C, \quad \tau \mapsto \frac{1}{2}\sum_{Q \in \mathcal{Q}_D}\sgn(Q_\tau)Q(\tau,1)^{k-1}\beta\left(\frac{Dv^2}{|Q(\tau,1)|^2};k - \frac{1}{2},\frac{1}{2}\right)\]
with exceptional set $E_D$ that is mapped to $f_{k,D}$ under both the Bol operator and the shadow operator (see Subsection \ref{subsec:DifferentialOperators} for a definition) which is impossible for a globally defined Maa{\ss} form with cuspidal shadow due to growth conditions. Here,
\[\beta(x;r,s) := \int_0^x t^{r-1}(1-t)^{s-1}\mathrm{d}t\]
is the incomplete $\beta$-function.

In a recent preprint \cite{brimo}, Bringmann and Mono constructed a new locally harmonic Maa{\ss} form
\[\mathcal{G}_{-k,D}: \H \to \C, \quad \tau \mapsto \frac{1}{2}\sum_{Q \in \mathcal{Q}_D}Q(\tau,1)^k\beta\left(\frac{Dv^2}{|Q(\tau,1)|^2};k + \frac{1}{2},\frac{1}{2}\right)\]
along the lines of Bringmann, Kane, and Kohnen's \cite{bkk} first example $\mathcal{F}_{1-k,D}$ as the preimage of $g_{k+1,D}$ under both the Bol operator and the shadow operator.

\subsection{Cycle integrals}

Löbrich and Schwagenscheidt \cite{lsmeromorphic} proved that $\mathcal{F}_{1-k,D}$ may be written as a trace of cycle integrals of Petersson's Poincaré series \cite{pet40}
\begin{align*}
    \mathbb{P}_{2k}(z,\tau) &\coloneqq v^{2k-1} \sum_{\gamma \in \Gamma}\frac{1}{(z - \tau)(z - \overline{\tau})^{2k-1}}\Big\vert_{2-2k,\tau}\gamma\\
    &= v^{2k-1} \sum_{\gamma \in \Gamma}\frac{1}{(z - \tau)(z - \overline{\tau})^{2k-1}}\Big\vert_{2k,z}\gamma
\end{align*}
which in turn was previously investigated by Bringmann and Kane \cite{brika20}. Concretely, Löbrich and Schwagenscheidt obtained the representation
\[\mathcal{F}_{1-k,D}(\tau) = \frac{(-1)^k}{4}\binom{2k - 2}{k - 1}\sum_{Q \in \mathcal{Q}_D \slash \Gamma}\int_{\Gamma_Q\backslash S_Q}\mathbb{P}_{2k}(z,\tau)Q(z,1)^{k-1}\textup{d}z\]
for $\tau \in \H \setminus E_D$ where $S_Q := \{\tau \in \H: Q_\tau = 0\}$ and $\Gamma_Q$ is the stabilizer of $Q$ under the action of $\Gamma$. In \cite[Theorem 1.1]{mo21}, Mono showed that an analogous result can be derived for $g_{k+1,D}$ by switching $z$ and $\tau$ and adjusting $k$ in the sense that
\begin{multline}\label{eq:ResultMono}
g_{k+1,D}(\tau) = \frac{2\Gamma(2k + 2)\prod_{j = 0}^{k-1}\left(k + j + 1)\right)}{D^{\frac{k+1}{2}}\Gamma\left(\frac{k+1}{2}\right)^2\prod_{l = 1 - k}^{-1}\left[(1 + l)(-l) + k(k + 1)\right]}\\ \times
\sum_{Q \in \mathcal{Q}_D \slash \Gamma} \int_{\Gamma_Q \backslash S_Q} \mathbb{P}_{2k+2}(\tau,z) Q(z,1)^{-k-1} \mathrm{d}z
\end{multline}
for $\tau \in \H \setminus E_D$. In \cite[subchapter 3.2]{lsmeromorphic}, Löbrich and Schwagenscheidt introduced a generalization of Petersson's Poincar\'e series $\mathbb{P}_{2k}$ which will be the star in our main result alongside another variant.

\begin{defn}\label{def:GeneralPeterssonPoincare}
    For $z, \tau \in \H$, $k_1 \in \N_{\geq 2}$ and $k_2 \in \Z$, we set
    \begin{align*}
        H_{k_1,k_2}(z,\tau) &:= \sum_{\gamma \in \Gamma} \frac{v^{k_1+k_2}}{(z - \tau)^{k_1-k_2}(z - \overline{\tau})^{k_1+k_2}}\Big|_{2k_1,z}\gamma\\
        &= \sum_{\gamma \in \Gamma} \frac{v^{k_1+k_2}}{(z - \tau)^{k_1-k_2}(z - \overline{\tau})^{k_1+k_2}}\Big|_{-2k_2,\tau}\gamma.
    \end{align*}
    Likewise, we set
    \begin{align*}
    \mathscr{P}_{k_1,k_2}(z,\tau) &\coloneqq \sum_{\gamma \in \Gamma} \frac{v^{k_1+k_2}}{\left(z+\tau\right)^{k_1-k_2}\left(z+\overline{\tau}\right)^{k_1+k_2}}\Big\vert_{2k_1,z}\gamma\\
    &= \sum_{\gamma \in \Gamma} \frac{v^{k_1+k_2}}{\left(z+\tau\right)^{k_1-k_2}\left(z+\overline{\tau}\right)^{k_1+k_2}}\Big\vert_{-2k_2,\tau}\gamma.
    \end{align*}
\end{defn}

Both variants transform like modular forms of weight $2k_1$ in $z$ and of weight $-2k_2$ in $\tau$ for $\Gamma$. The case where $H_{k_1,k_1-1}$ appears in the results by Löbrich and Schwagenscheidt \cite{lsmeromorphic} and Mono \cite{mo21}\footnote{Depending on whether $k_1$ is odd or even.}.

Our main result shows that $f_{k,D}$ may be written as a trace of cycle integrals with $\mathscr{P}_{k_1,k_2}(z,\tau)$ as integrands for certain $k_1, k_2 \in \Z$. Moreover, that this representation as well as the one given by Mono in \eqref{eq:ResultMono} is not unique.

\begin{thm}\label{thm:NewRepresentations}
    Let $k \in 2\N_{\geq 2}$. We have
    \begin{multline*}
        g_{k+1,D}(\tau) = \frac{2\Gamma(m + 1)\Gamma(2k + 2)\prod_{l = 0}^{k-1}\left[k + j + 1\right]}{D\Gamma\left(\frac{k+1}{2}\right)^2\prod_{l = 1 - k}^{-1}\left[(1 + l)(-l) + k(k + 1)\right]}\left(\prod_{l = 0}^{\frac{m}{2} - 1}\frac{(l + 1)}{(k - l)}\right)\\
        \times \sum_{Q \in \mathcal{Q}_D/\Gamma} \mathcal{C}_{2(m-k)}\left(H_{k+1,k-m}(\tau,\, \cdot \,),Q\right)
    \end{multline*}
    for $m \in 2\N$ with $m \leq 2k$ and
    \begin{multline*}
        f_{k,D}(\tau) = \frac{(-1)^{k+\frac{m}{2}}k}{D^{\frac{3k}{4}}} \binom{2k}{k} \left(\prod_{l = 0}^{\frac{m}{2}-1}\frac{2(k + l) + 1}{2l + 1}\right) \sum_{Q \in \mathcal{Q}_D/\Gamma}\mathcal{C}_{2(k+m)}\left(\mathscr{P}_{k,-(k+m)}(\tau,\, \cdot \,), Q\right)
    \end{multline*}
    for each $m \in 2\N$. Here, $\mathcal{C}_\kappa(F,Q)$ is the weight $\kappa$ cycle integral of $F$ with respect to $Q$ as defined in Subsection \ref{subsec:BinaryAndCycleIntegrals}.
\end{thm}

\subsection{Decomposing $g_{k+1,D}$ and its relation to quantum modular forms} In 2010, Zagier \cite{zagier10} introduced the notion of quantum modular forms as an extension of modular forms to the rational numbers. A function $f: \Q \to \C$ is called a quantum modular form if for every matrix $\gamma \in \Gamma$ the function
\[h_\gamma(x) := f(x) - \left(f\big\vert_k\right)(x)\]
can be analytically continued to $\R \setminus S_\gamma$ where $S_\gamma$ is a finite set.

In \cite[Lemma 2.1]{mo21}, Mono provided the decomposition
\[g_{k+1,D}(\tau) = p_{k+1,D}(\tau) - 2P_{g_{k+1,D}}(\tau)\]
with
\[p_{k+1,D}(\tau) \coloneqq \sum_{Q \in \mathcal{Q}_D} \frac{\sgn(Q)}{Q(\tau,1)^{k+1}}\]
being Parson's modular integral and
\[P_{g_{k+1,D}}(\tau) \coloneqq \sum_{Q \in \mathcal{Q}_D} \frac{\sgn(Q)\mathbf{1}_Q(\tau)}{Q(\tau,1)^{k+1}} = 2\sum_{\substack{Q=[a,b,c] \in \mathcal{Q}_D \\ a > 0}} \frac{\mathbf{1}_Q(\tau)}{Q(\tau,1)^{k+1}}\]
being a rational function where $\mathbf{1}_Q$ is the indicator function of the bounded area described by $S_Q$ and the real line.

We can investigate the error to modularity of
\[\lim_{\tau \to x} P_{g_{k+1,D}}(\tau) = -2\sum_{\substack{Q = [a,b,c] \in \mathcal{Q}_D \\ a < 0 < Q(x,1)}}\frac{1}{Q(x,1)^{k+1}}, \qquad x \in \Q,\]
and thus offer the following lemma.

\begin{prop}\label{prop:QuantumForm}
    The modular transformation behavior of $P_{g_{k+1,D}}$ is given by
    \[P_{g_{k+1,D}}(\tau + 1) = P_{g_{k+1,D}}(\tau) \quad \text{and} \quad P_{g_{k+1,D}}\left(\frac{-1}{\tau}\right) = \tau^{2k+2}\left(P_{g_{k+1,D}}(\tau) - q_{k+1,D}(\tau)\right)\]
    for $\tau \in \H$. Moreover, the function $x \mapsto \lim_{\tau \to x} P_{g_{k+1,D}}(\tau)$ for $x \in \Q$ is a weak quantum modular form of weight $2k+2$ with Haupttypus for the whole modular group.
\end{prop}

\section{Preliminaries}

\subsection{Binary integral quadratic forms and cycle integrals}\label{subsec:BinaryAndCycleIntegrals} A binary integral quadratic form $Q(X,Y) = aX^2 + bXY + cY^2$ will be denoted by $Q = [a,b,c]$ throughout. There exists an intensively studied action of $\Gamma$ on $\mathcal{Q}_D$ which is given by
\[\left(Q \circ \begin{pmatrix}
    a_1 & a_2 \\ a_3 & a_4
\end{pmatrix}\right)(X,Y) = Q(a_1X + a_2Y, a_3X + a_4Y)\]
and harmonizes with the action of $\Gamma$ on $\H$ via
\[(Q \circ \gamma)(\tau,1) = j(\gamma,\tau)^2Q(\gamma\tau,1)\]
where $j(\gamma,\tau) := (a_3\tau + a_4)$ is the usual modular cocycle for $\gamma = \begin{pmatrix}
    a_1 & a_2 \\ a_3 & a_4
\end{pmatrix} \in \Gamma$. The semi-circles $S_Q$ are often called \emph{Heegner geodesics} and connect the two real roots of the quadratic polynomial $Q(\tau,1)$. If $f$ is a smooth function on $\H$ transforming like a weight $k$ modular form, the weight $k$ cycle integral of $f$ with $Q$ is defined as
\[\mathcal{C}_k(f,Q) := D^{\frac{1}{2}-\frac{k}{4}}\int_{\Gamma_Q\backslash S_Q}f(z)Q(z,1)^{\frac{k}{2}-1}\mathrm{d}z\]
where the orientation of the integral is clockwise if
\[0 < \sgn([a,b,c]) := \begin{cases}
    \sgn(a), &a \neq 0\\
    \sgn(c), &a = 0
\end{cases}\]
and counter-clockwise otherwise.

\subsection{Locally harmonic Maa{\ss} forms}\label{subsec:LocallyHarmonicMaass} Let $k \in \Z$, $f$ a function defined on a subset of $\H$, and $\gamma = \begin{pmatrix}
    a_1 & a_2 \\ a_3 & a_4
\end{pmatrix}\in \Gamma$, then the operation of the \emph{weight $k$ slash operator} of $\gamma$ on $f$ is defined via
\[(f\big\vert_k\gamma)(\tau) := j(\gamma,\tau)^{-k}f\left(\frac{a_1\tau + a_2}{a_3\tau + a_4}\right).\]
Let $N \in \N$. The function $f$ is called a \emph{harmonic Maa{\ss} form for}
\[\Gamma_0(N) := \left\{\begin{pmatrix}
        a_1 & a_2 \\ a_3 & a_4
    \end{pmatrix}
\in \Gamma : a_4 \equiv 0 \mod N \right\}\]
\emph{of weight $k$} if
\begin{enumerate}
    \item for $\tau \in \H$, we have
    \[(f\big\vert_k\gamma)(\tau) = f(\tau)\]
    for all $\gamma \in \Gamma_0(N)$,
    \item for $\Delta_k = -v^2 \left(\frac{\partial^2}{\partial u^2} + \frac{\partial^2}{\partial v^2}\right) + ikv \left(\frac{\partial}{\partial u} + i\frac{\partial}{\partial v}\right)$ being the weight $k$ hyperbolic Laplacian operator, we have $\Delta_k f \equiv 0$ on $\H$,
    \item there exists a polynomial $P_f(\tau) \in \mathbb{C}[q^{-1}]$ such that
    \[f(\tau) - P_f(\tau) \in O(\exp(-\varepsilon v))\]
    for an $\varepsilon > 0$ for $v \to \infty$, and analogous condition hold at every other cusp of $\Gamma_0(N)$.
\end{enumerate}
The vector space of all such functions is denoted by $H_k(N)$.

The function $f$ is called a \emph{locally harmonic Maa{\ss} form for $\Gamma_0(N)$ of weight $k$} if there exists an exceptional set $E \subset \H$ of measure 0 such that
\begin{enumerate}
    \item for $\tau \in \H \setminus E$, we have
    \[(f\big\vert_k\gamma)(\tau) = f(\tau)\]
    for all $\gamma \in \Gamma_0(N)$,
    \item for $\Delta_k = -v^2 \left(\frac{\partial^2}{\partial u^2} + \frac{\partial^2}{\partial v^2}\right) + ikv \left(\frac{\partial}{\partial u} + i\frac{\partial}{\partial v}\right)$ being the weight $k$ hyperbolic Laplacian operator, we have $\Delta_k f \equiv 0$ on $\H \setminus E$,
    \item for $\tau \in E$, we have
    \[\lim_{\varepsilon \searrow 0} \frac{1}{2}\left(f(\tau + i\varepsilon) + f(\tau - i\varepsilon)\right) = f(\tau),\]
    \item the growth of $f$ towards $i\infty$ is bounded by a polynomial.
\end{enumerate}
If a (locally) harmonic Maa{\ss} form vanishes towards infinity, it is called \emph{cuspidal} or a \emph{cusp form}.

An overview can be found in \cite{bfkr} for example.

\subsection{Derivations of Petersson's Poincaré series}

For the sake of completeness, we briefly discuss the bivariate Poincaré series $H_{k_1,k_2}$ and $\mathscr{P}_{k_1,k_2}$ from the introduction. As the case of $H_{k_1,k_2}$ was implicitly done by Löbrich and Schwagenscheidt \cite{lsmeromorphic}, we omit it and leave the adjustments for the reader to apply. 

\begin{lemma}\label{lem:PeterssonPoincareSeriesConvergence}
	Let $z, \tau \in \H$ such that $z \not \in -\overline{\Gamma\tau}$, $k_1,k_2\in \Z, \gamma = \left(\begin{smallmatrix}
	a_1 & a_2 \\ a_3 & a_4
	\end{smallmatrix}\right)\in \Gamma$ and
	\[F(z,\tau) := \frac{v^{k_1+k_2}}{(z + \tau)^{k_1 - k_2}(z + \overline{\tau})^{k_1+k_2}}.\]
	Then $\left|F(z,\tau)\big\vert_{-2k_2,\tau}\gamma\right| \leq C\left|a_3\tau + a_4\right|^{-2k_1}$ for some $C>0$. Consequently, $\sum_{\gamma \in \Gamma}F(z,\tau)\big\vert_{-2k_2,\tau}\gamma$ converges absolutely for $k_1 \geq 2$.
\end{lemma}

\begin{proof}
	First, we point out that we have $|z+\gamma\tau| \geq |z + \overline{\gamma\tau}| > 0$ due to our assumptions. Then, we use that $\Im(\gamma\tau) = \frac{v}{|a_3\tau + a_4|^2}$ to get
	\[\left|F(z,\tau)\big\vert_{-2k_2,\tau}\gamma\right| \leq \left|\frac{v^{k_1+k_2}(a_3\tau + a_4)^{-2k_1}}{(z + \overline{\gamma\tau})^{2k_1}}\right| \leq C\left|a_3\tau + a_4\right|^{-2k_1}\]
	where we set $C := \frac{v^{k_1+k_2}}{M^{2k_1}}$ for $M:= \min\{|z+\overline{\gamma\tau}|:\gamma \in \Gamma\}$. The absolute convergence now follows directly.
\end{proof}

Having proven the absolute convergence, we now deal with the claimed modular transformation behaviour in either variable. In fact, this follows by the following simple observation.

\begin{lemma}\label{lem:PoincareModularTransformationEquality}
	In the setting of Lemma \ref{lem:PeterssonPoincareSeriesConvergence}, we have
	\[F(z,\tau)\big\vert_{-2k_2,\tau}\left(\begin{smallmatrix}
	a_1 & a_2 \\ a_3 & a_4
	\end{smallmatrix}\right) = F(z,\tau)\big\vert_{2k_1,z}\left(\begin{smallmatrix}
	a_4 & a_2 \\ a_3 & a_1
	\end{smallmatrix}\right).\]
	Consequently, $\mathscr{P}_{k_1,k_2}$ is well-defined for $k_1 \geq 2$.
\end{lemma}

We omit the proof as it is a straight-forward computation but not without mentioning that this differs to the case of $H_{k_1,k_2}$ which relies on the more obvious bijection $\gamma \mapsto -\gamma^{-1}$.

\subsection{Differential operators and Eichler integrals}\label{subsec:DifferentialOperators} Let $k, N \in \N$ and
\[\mathbb{D} := \frac{\partial}{2\pi i\partial \tau}\]
be a differential operator. The mapping
\[\mathbb{D}^{k-1}: H_{2-k}(\Gamma_0(N)) \to M_k^!(\Gamma_0(N)), \qquad f \mapsto \mathbb{D}^{k-1}(f)\]
is known as the \emph{Bol operator}. Likewise, the mapping
\[\xi_{2-k}: H_{2-k}(\Gamma_0(N)) \twoheadrightarrow S_k(\Gamma_0(N)), \qquad f \mapsto 2iv^{2-k}\overline{\frac{\partial}{\partial \overline{\tau}}f}\]
is referred to as the \emph{shadow operator}. Each of these two operators is accompanied by the so-called \emph{holomorphic} or \emph{non-holomorphic Eichler integral}, respectively, which are defined for an $f \in S_k(\Gamma_0(N))$ as
\begin{align*}
    \mathcal{E}_f(\tau) &\coloneqq -\frac{(2\pi i)^{k-1}}{(k-2)!}\int_\tau^{i\infty}f(\omega)(\tau - \omega)^{k-2} \, \textup{d}\omega \quad \textup{and}\\
    f^*(\tau) &\coloneqq (2i)^{1-k}\int_{-\overline{\tau}}^{i\infty}\overline{f(-\overline{\omega})}(\omega + \tau)^{k-2} \, \textup{d}\omega.
\end{align*}
They satisfy the relations
\[\mathbb{D}^{k-1}(\mathcal{E}_f) = f \quad \textup{and} \quad \xi_{2-k}(f^*) = f\]
as well as
\[\mathbb{D}^{k-1}(f^*) = 0 \quad \textup{and} \quad \xi_{2-k}(\mathcal{E}_f) = 0.\]

Each of these four operators maybe extended to locally harmonic Maa\ss forms in a natural way. The Bol and the shadow operator originate from two operators that also preserve modularity but raise or lower the weight of a modular form by 2 instead of changing it from $2-k$ to $k$: the \emph{{Maa\ss} raising and lowering operators}.

\begin{defn}\label{defn:RaisingLoweringOperator}
    Let $N, k \in \Z$. Then the \emph{weight $k$ {Maa\ss} raising operator} is defined via
    \[R_k: H_k(\Gamma_0(N)) \mapsto H_{k+2}(\Gamma_0(N)), \qquad f \mapsto 2i\frac{\partial}{\partial \tau}f + \frac{kf}{v}\]
    and the \emph{weight $k$ {Maa\ss} lowering operator} via
    \[L_k: H_k(\Gamma_0(N)) \mapsto H_{k-2}(\Gamma_0(N)), \qquad f \mapsto -2iv^2\frac{\partial}{\partial\overline{\tau}}f.\]
\end{defn}

\begin{rmk}
    Both {Maa\ss} raising and lowering operator commute with the weight $k$ slash operator in the sense that for any $f \in H_k(\Gamma_0(N))$ and $\gamma \in \Gamma$, we have
    \[L_kf\Big|_{k-2} \gamma = L_k \left(f\Big|_k \gamma \right) \quad \text{and} \quad R_kf\Big|_{k+2} \gamma = R_k \left(f\Big|_k \gamma \right).\]

    Furthermore, we have the obvious relation
    \[\xi_k = -v^{k-2}\overline{L_k}\]
    and the somewhat more subtle but also more famous identity
    \[\mathbb{D}^{k-1} = (-4\pi)^{1-k}R_{2-k}^{k-1}\]
    which is known as \emph{Bol's identity}\footnote{Which gave rise to the name of the Bol operator.}. Here, $R_{2-k}^{k-1} := R_{k-2} \circ \ldots \circ R_{2-k}$ is the $k-1$ times iterated raising operator where we adjust the weight in each step. If not clarified by the context, we will indicate the variable with respect to which a operator is considered by an index.
\end{rmk}

\section{Zagier's $f_{k,D}$ function}

We define the usual holomorphic Poincar{\'e} series of exponential type ($\Gamma \coloneqq \mathrm{SL}_2(\Z)$)
\begin{align*}
P_{k,m}(\tau) \coloneqq \sum_{\gamma \in \Gamma_{\infty} \backslash \Gamma} \left(q^m \big\vert_{k}\gamma\right)(\tau), \qquad k > 2, \qquad m \in \N,
\end{align*}
where $q := \exp(2\pi i \tau)$. One can check that $P_{k,m} \in S_k(\Gamma)$ for such choices of parameters $k$ and $m$, see \cite{iwa97}*{Proposition 3.2} for example, and that $P_{k,m}$ has the following property.
\begin{lemma}[\protect{\cite{iwa97}*{Theorem 3.3}}] \label{lem:PeterssonCoefficientFormula}
Let $f(\tau) = \sum_{n \geq 1} a_f(n)q^n \in S_k(\Gamma)$, and $\langle\cdot,\cdot\rangle$ be the usual Petersson inner product. Then, we have
\begin{align*}
    \langle f,P_{k,m}\rangle = \frac{\Gamma(k-1)}{(4\pi m)^{k-1}}a_f(m).
\end{align*}
\end{lemma}
In fact, this property characterizes $P_{k,m}$, cf.\ \cite{koh85}*{eq. (4)}, and is sometimes attributed to Petersson as the \textit{Petersson coefficient formula}.

The idea is to combine Lemma \ref{lem:PeterssonCoefficientFormula} with its ``hyperbolic analogue'', which reads as follows.
\begin{lemma}[\protect{\cite{koza84}*{Section 3}, \cite{koh85}*{Proposition 7}}] \label{lem:HyperbolicPeriods}
Let $f \in S_{2k}(\Gamma)$. Then, we have
\begin{align*}
    \langle f, f_{k,D} \rangle = \pi \binom{2k-2}{k-1}2^{3-2k}D^{-\frac{k}{2}} \sum_{Q \in \mathcal{Q}_D \slash \Gamma} \mathcal{C}_{2k}(f,Q).
\end{align*}
\end{lemma}
Note that the cycle integrals on the right hand side are called {\it hyperbolic periods} sometimes.

If $D > 0$ and $k > 2$, Zagier \cite{zagier75}*{Appendix 2} proved that $f_{k,D} \in S_{2k}(\Gamma)$ and calculated its Fourier coefficients $a_{f_{k,D}}(n)$ in terms of quadratic \emph{Weyl} sums (special Kloosterman sums) and the usual $J$-Bessel function. 
\begin{lemma} \label{lem:fkDfourier}
If $D > 0$ and $k > 2$, then the Fourier coefficients of the function $f_{k,D}$ are real.
\end{lemma}

\begin{proof}
We observe that
\begin{align*}
\overline{f_{k,D}(-\overline{\tau})} = \overline{\left(\sum_{b^2-4ac = D} \frac{1}{\left(a(-\overline{\tau})^2+b(-\overline{\tau})+c\right)^k}\right)} = \sum_{b^2-4ac = D} \frac{1}{\left(a\tau^2-b\tau+c\right)^k} = f_{k,D}(\tau),
\end{align*}
where we mapped $b \mapsto -b$ in the last step. Clearly, we have $\overline{e^{2\pi i m (-\overline{\tau})}} = e^{2\pi i m \tau}$, and hence $\overline{a_{f_{k,D}}(n)} = a_{f_{k,D}}(n)$ for every $n \in \N$.
\end{proof}

By Lemma \ref{lem:PeterssonCoefficientFormula}, we obtain
\begin{align*}
    a_{f_{k,D}}(m) = \frac{(4\pi m)^{2k-1}}{\Gamma(2k-1)} \langle f_{k,D},P_{2k,m}\rangle,
\end{align*}
and we deduce from Lemma \ref{lem:fkDfourier} that
\begin{align*}
\langle f_{k,D},P_{2k,m}\rangle = \overline{\langle f_{k,D},P_{2k,m}\rangle} = \langle P_{2k,m}, f_{k,D}\rangle.
\end{align*}
Consequently, we obtain
\begin{align*}
    f_{k,D}(\tau) &= \sum_{m \geq 1} a_{f_{k,D}}(m) q^m = \frac{(4\pi)^{2k-1}}{\Gamma(2k-1)} \sum_{m \geq 1} m^{2k-1} \langle f_{k,D},P_{2k,m}\rangle q^m \\
				&= \frac{(4\pi)^{2k-1}}{\Gamma(2k-1)} \sum_{m \geq 1} m^{2k-1} \langle P_{2k,m}, f_{k,D}\rangle q^m.
\end{align*}
In other words, we use that $\{P_{2k,m} \colon m \in \N\}$ generates $f \in S_{2k}(\Gamma)$. Next, we utilize Lemma \ref{lem:HyperbolicPeriods} to rewrite the right hand side as
\begin{multline*}
    f_{k,D}(\tau) \\
= \pi \binom{2k-2}{k-1}2^{3-2k}D^{-\frac{k}{2}} \frac{(4\pi)^{2k-1}}{\Gamma(2k-1)} \sum_{m \geq 1} m^{2k-1} \sum_{Q \in \mathcal{Q}_D \slash \Gamma} \mathcal{C}_{2k}(P_{2k,m},Q)q^m \\
        = \frac{2 \cdot (2\pi)^{2k}D^{\frac{1}{2}-k}}{\Gamma(k)^2} \sum_{Q \in \mathcal{Q}_D \slash \Gamma} \int_{\Gamma_Q \backslash S_Q} \left(\sum_{m \geq 1} m^{2k-1} P_{2k,m}(z)q^m \right) Q(z,1)^{k-1}\mathrm{d}z.
\end{multline*}

The generating function of $P_{2k,m}$ inside the cycle integral is known to be a specific case of $\mathscr{P}_{k_1,k_2}$.

\begin{lemma} \label{lem:poincaresummation}
For every $z$, $\tau \in \H$ and $k > 1$, it holds that
\begin{align*}
\frac{(2\pi i)^{2k}}{(2k-1)!} \sum_{m \geq 1} m^{2k-1} P_{2k,m}(z)q^m = \mathscr{P}_{k,-k}(z,\tau).
\end{align*}
\end{lemma}

\begin{proof}
We recall the Lipschitz summation formula
\begin{align*}
\sum_{n \in \Z} \frac{1}{(w+n)^{\kappa}} = \frac{(-2\pi i)^{\kappa}}{(\kappa-1)!}\sum_{m\geq 1} m^{\kappa-1} e^{2\pi i m w},
\end{align*}
valid for $\kappa \in \N_{\geq 2}$, $w \in \H$, and which can be found in \cite{123}*{p.\ $16$} for instance. We deduce that
\begin{align*}
\mathscr{P}_{k,-k}(z,\tau) &= \sum_{\gamma \in \Gamma \slash\Gamma_{\infty}} \sum_{n \in \Z} \frac{1}{\left((z+n)+\tau\right)^{2k}}\Big\vert_{{2k},z}\gamma = \frac{(-2\pi i)^{2k}}{(2k-1)!} \sum_{\gamma \in \Gamma \slash\Gamma_{\infty}} \sum_{m\geq 1} m^{2k-1} e^{2\pi i m (z+\tau)} \Big\vert_{{2k},z}\gamma \\
&= \frac{(2\pi i)^{2k}}{(2k-1)!} \sum_{m \geq 1} m^{2k-1} P_{2k,m}(z)q^m,
\end{align*}
as claimed.
\end{proof}

\begin{rmks}
\
\begin{enumerate}
\item This summation formula is the ``parabolic analogue'' of the hyperbolic summation formula
\begin{align*}
\Omega_{k}(\tau,z) = \sum_{D > 0} D^{k-\frac{1}{2}} f_{k,D}(z)q^D.
\end{align*}
Kohnen and Zagier \cites{koza81, koza84} observed that $\Omega_{k}$ is the theta kernel function for the Shimura \cite{shim} and Shintani \cite{shin} lift. Kohnen \cite{koh85} generalized this observation to odd and square-free level. More recently, Ueda and Yamana \cite{ueya} extended Kohnen's results to even and square-free level.
\item Bringmann and Kane \cite{brika20}*{Section $3$} generalized Lemma \ref{lem:poincaresummation} to the generating function of Maa{\ss}--Poincar{\'e} series, which yields an alternative proof of Lemma \ref{lem:poincaresummation}. To be more precise, one may apply the operator $\xi_{2-2k}$ with respect to $z$ to the identities \cite{brika20}*{equation (3.9), Proposition 3.3}, and note that the functions in both equations coincide. Then, we apply \cite{brika20}*{(2.4), (3.7)} to rewrite each side of the resulting equation. However, this needs the restricition $\im(\tau) > \max\big\{\im(z), \frac{1}{\im(z)}\big\}$ and the identity
\begin{align*}
\mathscr{P}_{k,-k}(z,-\overline{\tau}) = \overline{\mathscr{P}_{k,-k}(-\overline{z},\tau)}.
\end{align*}
\end{enumerate}
\end{rmks}

Finally, we check the following claim.
\begin{lemma}\label{lem:CurlyPCuspForm}
The function $z \mapsto \mathscr{P}_{k,-k}(z,\tau)$ is a cusp form for every $k \geq 2$, and we have $\mathscr{P}_{k,-k}(z,\tau) = \mathscr{P}_{k,-k}(\tau,z)$.
\end{lemma}

\begin{proof}
The first assertion follows by Lemma \ref{lem:poincaresummation}, because $P_{k,m} \in S_{k}(\Gamma)$ for every $k \geq 4$ even and every $m \in \N$. The second assertion follows by Lemma \ref{lem:PoincareModularTransformationEquality}.
\end{proof}

We may now obtain a first representation of $f_{k,D}$ as a trace of cycle integrals.

\begin{thm} \label{thm:fkDresult}
Let $k > 2$ and $\mathscr{P}_{k,-k}$ be as in Definition \ref{def:GeneralPeterssonPoincare}. Then,
\begin{align*}
f_{k,D}(\tau) = (-1)^{k} D^{\frac{1}{2} - k}k \binom{2k}{k} \sum_{Q \in \mathcal{Q}_D \slash \Gamma} \int_{\Gamma_Q \backslash S_Q} \mathscr{P}_{k,-k}(z,\tau) Q(z,1)^{k-1}\mathrm{d}z.
\end{align*}
\end{thm}

\begin{proof}
    We simply combine the previous observations and obtain our desired result.
\end{proof}

Theorem \ref{thm:fkDresult} has a nice application, which we prepare with another lemma.
\begin{lemma}[\protect{\cite{koekri}*{exercises to Section III.2}}] \label{lem:exercise}
If $k$ is not a multiple of $\vt{\Gamma_\tau}$ then $\mathscr{P}_{k,-k}(z,\tau)$ vanishes identically.
\end{lemma}

\begin{proof}
Let $\tau \in \H$. By Definition \ref{def:GeneralPeterssonPoincare} and Lemma \ref{lem:CurlyPCuspForm}, we obtain
\begin{align*}
\mathscr{P}_{k,-k}(z,\tau) &= \sum_{\gamma \in \Gamma} \frac{1}{\left(z+\tau\right)^{2k}}\Big\vert_{2k,\tau}\gamma  = \left(\sum_{\gamma_1 \in \Gamma\slash \Gamma_{\tau}}\frac{1}{\left(z+\tau\right)^{2k}}\Big\vert_{2k,\tau}\gamma_1\right)\left(\sum_{\gamma_2 \in \Gamma_{\tau}} 1\Big\vert_{2k,\tau}\gamma_2\right).
\end{align*}
The only cases where $k$ may fail to be a multiple of $|\Gamma_\tau|$ are $z \in \Gamma\rho$ where $\rho := \exp \left(\frac{2\pi i}{3}\right)$ and due to the modularity property of $\mathscr{P}_{k,-k}(\tau,z)$ and $|\Gamma_\rho| = 3$, the claim follows if $\mathscr{P}_{k,-k}(\tau,\rho) = 0$ for $3 \nmid k$. We simply evaluate the right factor of the above equation and retrieve:
\begin{align*}
    \sum_{\gamma \in \Gamma_\rho} 1 \Big\vert_k \gamma &= 1^{-k} + (\rho - 1)^{-k} + \rho^{-k} = 1^k + (\overline{\rho} - 1)^k + \overline{\rho}^k
\end{align*}
which vanishes if $3 \nmid k$.    
\end{proof}

This established, we obtain a new proof for the following result.
\begin{cor} \label{cor:fkDzeros}
If $w$ is $\Gamma$-equivalent to $e^{\frac{\pi i}{3}}$ (or to $e^{\frac{2\pi i}{3}}$) and if $k$ is not a multiple of $6$ then $f_{k,D}(w) = 0$.
\end{cor}

\begin{proof}
Recall that $k$ is assumed to be even, otherwise $f_{k,D}$ vanishes identically and the claim follows trivially. Moreover, the stabilizer of the point $e^{\frac{\pi i}{3}}$ has order $3$ (see \cite{koblitz}*{p.\ $102$} for example). Hence if $k$ is not divisible by $2$ and $3$ and $w$ is $\Gamma$-equivalent to $e^{\frac{\pi i}{3}}$ then $\mathscr{P}_{k,-k}(\cdot,w)$ vanishes identically by Lemma \ref{lem:exercise}, and the claim follows by Theorem \ref{thm:fkDresult}.
\end{proof}

\begin{rmk}
This might have implications on the divisor modular form of $f_{k,D}$, see \cite{bklor} for more details on this concept.
\end{rmk}

\section{Applying the work by Alfes, Bringmann, Löbrich, Mono, and Schwagenscheidt}

Recall the functions $H_{k_1,k_2}$ from Definition \ref{def:GeneralPeterssonPoincare} in the introduction. With respect to $\tau$, these series behave delightfully under $R_{-2k_2}$ and $L_{-2k_2}$ in the sense of

\begin{lemma}[\protect{\cite{lsmeromorphic}*{Lemma 3.2}}] \label{lem:PeterssonPoincareRaisingLowering}
    For $k_1 \in \N_{\geq 2}$ and $k_2 \in \Z$, it holds that
    \begin{align*}
        R_{-2k_2}(H_{k_1,k_2}(z,\tau)) &= (k_1 - k_2)H_{k_1,k_2-1}(z,\tau),\\
        L_{-2k_2}(H_{k_1,k_2}(z,\tau)) &= (k_1 + k_2)H_{k_1,k_2+1}(z,\tau).
    \end{align*}
\end{lemma}

These identities analogously extend to $\mathscr{P}_{k_1,k_2}$ in the sense that
\begin{align*}
	R_{-2k_2,\tau}\mathscr{P}_{k_1,k_2}(z,\tau) &= (k_1 - k_2)\mathscr{P}_{k_1,k_2-1}(z,\tau),\\
	L_{-2k_2,\tau}\mathscr{P}_{k_1,k_2}(z,\tau) &= (k_1 + k_2)\mathscr{P}_{k_1,k_2+1}(z,\tau)
\end{align*}
for $k_1 \in \N_{\geq 2}$ and $k_2 \in \Z$.

Using the well-known identity
\[\Delta_{r,\tau} = -R_{r-2}L_r, \qquad r \in \Z\]
we obtain that
\begin{align*}
    \Delta_{-2k_2,\tau}H_{k_1,k_2}(z,\tau) &= -(k_1 + k_2)(k_1 - k_2 - 1)H_{k_1,k_2}(z,\tau),\\
    \Delta_{-2k_2,\tau}\mathscr{P}_{k_1,k_2}(z,\tau) &= -(k_1 + k_2)(k_1 - k_2 - 1)\mathscr{P}_{k_1,k_2}(z,\tau).
\end{align*}
In particular, these functions are harmonic in $\tau$ if and only if $k_1 = - k_2$ or $k_1 = k_2 + 1$ which are the cases of $f_{k,D}$ by Theorem \ref{thm:fkDresult} and $\mathcal{F}_{1-k,D}$ due to
\[H_{k,k-1}(z,\tau) = \mathbb{P}_{2k}(z,\tau).\]
In these two cases, iteratively invoking Lemma \ref{lem:PeterssonPoincareRaisingLowering} yields simple relations, the first of which was already mentioned by Löbrich and Schwagenscheidt in \cite{lsmeromorphic}.

\begin{cor}\label{cor:IterativelyApplyingRandLtoHk_1k_2}
For $k, m \in \N_{\geq 1}$, it holds that
\begin{align}
    \label{eq:CorIARaLH_1}
    R_{-2k,\tau}^mH_{k+1,k}(z,\tau) &= \Gamma(m + 1)H_{k+1,k-m}(z,\tau);\\
    \label{eq:CorIARaLH_2}
    R_{2k,\tau}^m\mathscr{P}_{k,-k}(z,\tau) &= \frac{\Gamma(2k + m)}{\Gamma(2k)}\mathscr{P}_{k,-k-m}(z,\tau)
\end{align}
as well as
\begin{align*}
    L_{-2k,\tau}^mH_{k+1,k}(z,\tau) &= \frac{\Gamma(2k + m + 1)}{\Gamma(2k + 1)}H_{k+1,k+m}(z,\tau);\\
    L_{2k,\tau}^m\mathscr{P}_{k,-k}(z,\tau) &= 0.
\end{align*}
\end{cor}

We now exploit a theorem by Alfes and Schwagenscheidt that generalizes a previous result by Bringmann, Guerzhoy, and Kane \cites{BGK14,BGK15}.

\begin{thm}[\protect{\cite{ANS}*{Theorem 1.1}}] \label{thm:ANS}
    Let $F: \H \to \C$ be a smooth function which transforms like a modular form of weight $2 - 2\kappa, \kappa \in \Z,$ for $\Gamma$. Then the identity
    \[\mathcal{C}_{-2\kappa}(L_{2-2\kappa}F,Q) = \mathcal{C}_{4-2\kappa}(R_{2-2\kappa}F,Q) = \overline{\mathcal{C}_{2\kappa}(\xi_{2-2\kappa}F,Q)}\]
    holds.

    Furthermore, if $F$ is a weak Maa{\ss} form of weight $2 - 2\kappa$ with eigenvalue $\lambda$, we have
    \begin{align}\label{eq:ANLRaising}
        \mathcal{C}_{2-2\ell}(R_{2-2\kappa}^{\kappa-\ell}F,Q) &= \left((\kappa + \ell)(\kappa - \ell - 1) - \lambda\right) \mathcal{C}_{-2-2\ell}(R_{2-2\kappa}^{\kappa-\ell-2}F,Q), &\ell \leq \kappa - 2;\\
        \label{eq:ANLLowering}
        \mathcal{C}_{2\ell-2}(L_{2-2\kappa}^{2-\kappa-\ell}F,Q) &= \left((\kappa + \ell)(\kappa - \ell - 1) - \lambda\right) \mathcal{C}_{2+2\ell}(L_{2-2\kappa}^{-\kappa-\ell}F,Q), &\ell \leq -\kappa.
    \end{align}
    Here, we adjust the weight while iterating either operator as indicated in Section 2.
\end{thm}

This theorem is proven using representations of cycle integrals as integrals over an interval on the real line and thus may be applied in the case of $H_{k,k-1}$ since we explicitly exclude the cases where $z \in E_D$ and hence our path of integration does not cross any poles. Combining Corollary \ref{cor:IterativelyApplyingRandLtoHk_1k_2} and Theorem \ref{thm:ANS}, we obtain

\begin{lemma}\label{lem:FactorsForMainTheorem}
    Let $m \in 2\N$ and $k \in \N_{\geq 1}$. Then
    \begin{equation}\label{eq:FactorsSmallG}
        \mathcal{C}_{-2(k - m)}\left(H_{k+1,k-m}(z, \, \cdot \,), Q\right) = \left(\prod_{l = 0}^{\frac{m}{2} - 1}\frac{(k - l)}{(l + 1)}\right) \cdot \mathcal{C}_{-2k}\left(H_{k+1,k}(z, \, \cdot \,), Q\right)
    \end{equation}
    and
    \begin{equation}\label{eq:FactorsSmallF}
        \mathcal{C}_{2(k+m)}\left(\mathscr{P}_{k,-k-m}(z, \, \cdot \,), Q\right) = (-1)^{\frac{m}{2}}\left(\prod_{l = 0}^{\frac{m}{2}-1}\frac{2l + 1}{2(k + l) + 1}\right)\cdot \mathcal{C}_{2k}\left(\mathscr{P}_{k,-k}(z, \, \cdot \,), Q\right)
    \end{equation}
\end{lemma}

\begin{proof}
    We first prove \eqref{eq:FactorsSmallG} by applying \eqref{eq:CorIARaLH_1} and \eqref{eq:ANLRaising} with $\kappa = k + 1$, $\ell = k - m - 1$ to $F = H_{k+1,k}(z, \, \cdot \,)$ so that
    \begin{align*}
        \mathcal{C}_{-2(k-m-2)}&\left(H_{k+1,k-m-2}(z, \, \cdot \,), Q\right)\\
        &= \frac{1}{\Gamma(m + 3)}\mathcal{C}_{2-2(k-m-1)}(R_{-2k}^{m+2}H_{k+1,k}(z, \, \cdot \,), Q)\\
        &= \frac{(2k-m)(m+1)}{\Gamma(m + 3)}\mathcal{C}_{-2-2(k-m-1)}(R_{-2k}^{m}H_{k+1,k}(z, \, \cdot \,), Q)\\
        &= \frac{2k - m}{m + 2}\mathcal{C}_{-2(k-m)}\left(H_{k+1,k-m}(z, \, \cdot \,),Q\right)
    \end{align*}
    which yields the proof inductively.

    Likewise, we apply \eqref{eq:CorIARaLH_2} and \eqref{eq:ANLRaising} with $\kappa = 1 - k$, $\ell = - k - m - 1$ to $F = \mathscr{P}_{k,-k}(z, \, \cdot \,)$ so that
    \begin{align*}
        \mathcal{C}_{2(k+m+2)}&\left(\mathscr{P}_{k,-k-m-2}(z, \, \cdot \,), Q\right)\\
        &= \frac{\Gamma(2k)}{\Gamma(2k + m + 2)}\mathcal{C}_{2-2(-k-m-1)}\left(R_{-2k}^{m+2}\mathscr{P}_{k,-k}(z, \, \cdot \,), Q\right)\\
        &= \frac{\Gamma(2k)(-2k-m)(m+1)}{\Gamma(2k + m + 2)}\mathcal{C}_{-2-2(-k-m-1)}\left(R_{-2k}^{m}\mathscr{P}_{k,-k}(z, \, \cdot \,), Q\right)\\
        &= -\frac{(m+1)}{(2k + m + 1)}\mathcal{C}_{2(k+m)}\left(\mathscr{P}_{k,-k-m}(z, \, \cdot \,), Q\right)
    \end{align*}
    which proves \eqref{eq:FactorsSmallF} inductively, too.
\end{proof}

We now add the final ingredients: Mono's result \eqref{eq:ResultMono} and Theorem \ref{thm:fkDresult}.

\begin{proof}[Proof of Theorem \ref{thm:NewRepresentations}.]
    Referring to the definition, we first rewrite $\mathbb{P}_{2k+2}(\tau,z) = H_{k+1,k}(\tau,z)$ and then use our previous machinery with respect to $z$ instead of $\tau$.
    Lemma \ref{lem:FactorsForMainTheorem} proves the statement about $g_{k+1,D}$ up to $m = 2k$ because the base case $m = 0$ follows by virtue of \eqref{eq:ResultMono}. Analogously, we derive the statement for $f_{k,D}$ from Lemma \ref{lem:FactorsForMainTheorem} and Theorem \ref{thm:fkDresult}.
\end{proof}

\section{Decomposition of $g_{k+1,D}$ and its relation to quantum modular forms}

\subsection{Quantum modular forms} Due to the modular properties of $g_{k+1,D}$ and $p_{k+1,D}$, we can now show that the local part $P_{k+1,D}$ of Mono's function gives rise to a quantum modular form.

\begin{proof}[Proof of Proposition \ref{prop:QuantumForm}]
    Parson \cite[Theorem 3.1]{par93} showed that
    \[p_{k+1,D}(\tau + 1) = p_{k+1,D}(\tau) \qquad \text{and} \qquad p_{k+1,D}\left(\frac{-1}{\tau}\right) = \tau^{2k+2}\left(p_{k+1,D}(\tau) - 2q_{k+1,D}(\tau)\right)\]
    where
    \[q_{k+1,D}(\tau) := \sum_{\substack{Q=[a,b,c] \in \mathcal{Q}_D \\ b^2 < D}}\frac{\sgn(Q)}{Q(\tau,1)^{k+1}}\]
    for $\tau \in \H$. The transformation behavior is now a straightforward calculation by \cite[Theorem 1.1 (1) and Proposition 3.1 (1)]{mo21}.

    As mentioned above, the function
    \[P_{k+1,D}^\Q:\Q \to \C, \qquad x \mapsto \lim_{\tau \to x}P_{g_{k+1,D}}(\tau)\]
    is well-defined. In order for $P_{k+1,D}^\Q$ to be a weak quantum modular form, we have to be able to continue the function
    \[h_\gamma(x) := P_{k+1,D}^\Q(x) - (a_3 x + a_4)^{-2k-2}P_{k+1,D}^\Q(\gamma x)\]
    for all $\gamma = \begin{pmatrix}
        a_1 & a_2 \\ a_3 & a_4
    \end{pmatrix} \in \Gamma$ to a real analytic function on $\R \setminus S_\gamma$ where $S_\gamma$ is a finite subset of the real numbers. We may check this only for the generators $T$ and $S$ of $\Gamma$. Obviously, $h_T(x) \equiv 0$. Furthermore,
    \[h_S(x) = q_{k+1,D}(x)\]
    where we continued $q_{k+1,D}$ to the rationals. Since $q_{k+1,D}$ sums only finitely many terms, it is analytic on $\R \setminus M_{S,D}$ where
    \[M_{S,D} := \left\{\frac{-b \pm \sqrt{D}}{2a} \in \R: b^2 < D = b^2 - 4ac\right\}\]
    which, as $D$ is fixed, is a finite set. Thus, $h_S(x)$ is also analytic on $\R \setminus M_{S,D}$ and the claim follows.
\end{proof}

\begin{rmk}
    The set $M_{S,D}$ consists of irrational algebraic integers in $\Q\left(\sqrt{D}\right)$ whose norm is negative, i.e. of integers in $\mathcal{O}_{Q\left(\sqrt{D}\right)}$ that switch sign under the non-trivial automorphism in the Galois group of $\Q\left(\sqrt{D}\right)$ over $\Q$.
\end{rmk}

\subsection{Decomposition of $\mathcal{F}_{1-k,D}$ and $\mathcal{G}_{-k,D}$}
Two major results state that $\mathcal{F}_{1-k,D}$ and $\mathcal{G}_{-k,D}$ are sums of the holomorphic and non-holomorphic Eichler integrals of their respective shadow, i.e. $f_{k,D}$ and $g_{k+1,D}$. To be precise, we have by \cite[Theorem 1.3]{bkk} that
\[\mathcal{F}_{1-k,D}(\tau) = P_C(\tau) + D^{\frac{1}{2}-k}f_{k,D}^*(\tau) - D^{\frac{1}{2}-k}\frac{(2k - 2)!}{(4\pi)^{2k-1}}\mathcal{E}_{f_{k,D}}(\tau)\]
where $P_C$ is a local polynomial of degree at most $2k - 2$, depending only on the connected component $C$ of $\H \setminus E_D$ containing $\tau$ and by \cite[Theorem 1.2]{brimo} that
\begin{align}\label{eq:CurlyGDecomposition}
\mathcal{G}_{-k,D}(\tau) = c_\infty + D^{k+\frac{1}{2}}g^*_{k+1,D}(\tau) - \frac{D^{k+\frac{1}{2}}(2k)!}{(4\pi)^{2k+1}}\mathcal{E}_{g_{k+1,D}}(\tau)
\end{align}
where $c_\infty$ is a specific constant.

In view of the results by L{\"o}brich and Schwagenscheidt \cite{lsmeromorphic}*{Theorem 4.2} and by Mono \cite{mo21}*{Theorem 1.1} as well as of Theorem \ref{thm:fkDresult} it is natural to ask if $\mathcal{G}_{-k,D}$ admits a representation as a trace of cycle integrals as well. Extensively studying the case of $m = 0$ in Theorem \ref{thm:NewRepresentations} remained fruitless so far but other cases might produce new perspectives providing an easy-to-guess preimage of the seed of $H_{k+1,k-m}$ under the shadow and Bol operator.

\subsection{Fourier coefficients of $g_{k+1,D}$ and $\mathcal{G}_{-k,D}$} Due to the decomposition given in \eqref{eq:CurlyGDecomposition}, knowledge of the Fourier coefficients of $\mathcal{G}_{-k,D}$ is equivalent to that of the Fourier coefficients of $g_{k+1,D}$ via known relations, cf. \cite[Theorem 5.5 and Theorem 5.10]{bfkr}. We imitate Zagier's \cite[Appendix 2]{zagier75} and Parson's \cite[Theorem 3.1]{par93} approach via the Laplace inverse.

As the Fourier coefficients originating from $p_{k+1,D}$ were computed by Parson, we are left with those of $P_{g_{k+1,D}}$. To this end, define
\[\widetilde{g_{k + 1,D}}(\tau) := \sum_{Q \in \mathcal{Q}_D} \frac{\sgn(Q)^{k+1}\mathbf{1}_Q(\tau)}{Q(\tau,1)^{k+1}} = \sum_{a \in \Z} \sgn(a) \widetilde{g_{k+1,D}^a}(\tau)\]
where
\[\widetilde{g_{k+1,D}^a}(\tau) := \sum_{\substack{b \in \Z \\ b^2 \equiv D \pmod*{4a}}} \frac{\mathbf{1}_Q(\tau)}{(a\tau^2 + b\tau + \frac{b^2 - D}{4a})^{k+1}}.\]
Since $S_Q = S_{-Q}$ and $a \neq 0$, this simplifies to
\[\widetilde{g_{k + 1,D}}(\tau) = 2 \sum_{a = 1}^\infty \widetilde{g_{k+1,D}^a}(\tau).\]
The functions $\widetilde{g_{k+1,D}^a}$ can be rewritten because the tuple $(a,b)$ already determines the quadratic form. Furthermore, we transform the indicator function into a further summation condition. Thus,
\begin{multline*}
    \widetilde{g_{k+1,D}^a}(\tau) = \sum_{\substack{b \in \Z \\ b^2 \equiv D \pmod*{4a} \\ \left(\frac{2au + b}{2a}\right)^2 + v^2 < \frac{D}{4a^2}}} \frac{1}{(a\tau^2 + b\tau + \frac{b^2 - D}{4a})^{k+1}}\\
    =  \sum_{n \in \Z} \sum_{\substack{b \in \mathcal{R}_a \\ |b + 2a(u + n)| < \Re\left(\sqrt{D - 4a^2v^2}\right)}} \frac{1}{(a(\tau + n)^2 + b(\tau + n) + \frac{b^2 - D}{4a})^{k+1}}
\end{multline*}
where $\mathcal{R}_a$ is a fixed system of representatives of the residue classes of $\nicefrac{\Z}{2a\Z}$ whose square is congruent to $D$ modulo $4a$. However, merging this with Parson's result is not as straightforward as one might hope as this requires the Laplace inverse of the indicator function.

Nevertheless, the Fourier coefficients of $g_{k+1,D}$ share a property with those of $f_{k,D}$ in any connected component of $\H \setminus E_D$.

\begin{lemma}
    Let $k > 2$. Then, the Fourier coefficients of $g_{k+1,D}$ are real in every connected component of $\H \setminus E_D$.
\end{lemma}

\begin{proof}
    In a similar manner to the proof of Lemma \ref{lem:fkDfourier}, we have
    \[g_{k+1,D}(\tau) = \overline{g_{k+1,D}(-\overline{\tau})}\]
    after noting that
    \[[a,b,c]_{-\overline{\tau}} = [a,-b,c]_\tau\]
    for any $[a,b,c] \in \mathcal{Q}_D$.
\end{proof}

More interestingly, as pointed out by Mono in the proof of \cite[Theorem 1.1]{mo21}, the local part of the Fourier coefficients of $g_{k+1,D}$ vanishes above the net of Heegner geodesics and hence the local Fourier expansion of $g_{k+1,D}$ coincides with the one of Parson's function for $v$ sufficiently large. This allows us to also compute the Fourier coefficients of $\mathcal{G}_{-k,D}$ in that region, i.e.
\[\mathcal{G}_{-k,D}(\tau) = c_\infty + \frac{D^{k+\frac{1}{2}}}{(4\pi)^{2k+1}}\sum_{r\geq 1}\frac{c(r)}{r^{2k+1}}\Gamma(2k+1,4\pi rv)q^{-r} - \frac{D^{k+\frac{1}{2}}(2k)!}{(4\pi)^{2k+1}}\sum_{r\geq 1}\frac{c(r)}{r^{2k+1}}q^r\]
where
\begin{align*}
    c(r) &\coloneqq \frac{2^{k + \frac{3}{2}}\pi^{k+2}r^{k + \frac{1}{2}}}{D^{\frac{k}{2} + \frac{1}{4}} k!}\sum_{a = 1}^\infty \frac{S_a(D;r)}{\sqrt{a}}J_{k+\frac{1}{2}}\left(\frac{\pi r \sqrt{D}}{a}\right),\\
    c_\infty &\coloneqq \frac{\pi D^{k + \frac{1}{2}}}{2^{2k}(2k + 1)}\sum_{a \geq 1}\frac{S_a(D;0)}{a^{k+1}}
\end{align*}
are the Fourier coefficients of Parson's function \cite[Theorem 1.3]{par93} and the value of $\mathcal{G}_{-k,D}$ at infinity \cite[eq. (5.3)]{brimo}, respectively. Here,
\[S_a(D;r) \coloneqq \sum_{\substack{b \pmod*{2a} \\ b^2 \equiv D \pmod*{4a}}} \exp\left(2\pi i\frac{br}{2a}\right), \qquad a \in \N\]
are so-called \emph{Weyl} sums and $J_{k+\frac{1}{2}}$ is the usual Bessel function of the first kind.

\begin{rmk}
    One might hope that we are able to imitate Mono's proof which takes advantage of the fact that the $c(n)$ in the Fourier expansion of $g_{k+1,D}$ may be understood as cycle integrals of various Niebur-Poincar\'{e} series
    \[G_n(z,s) \coloneqq \sum_{\gamma \in \Gamma_\infty \backslash \Gamma}g_n(z,s)\big\vert_0\gamma\]
    where $\Re(s) > 1$ and
    \[g_n(z,s) \coloneqq \frac{\Gamma(s)}{\Gamma(2s)}M_{0,s - \frac{1}{2}}(4\pi|n|y)\exp(2\pi inx)\]
    with $M_{\mu,\nu}$ being the usual Whittaker function \cite[Definition 6.10]{bfkr} as was proven by Duke, Imamo\u{g}lu, and T\'{o}th \cite{dit11}. However, applying multiple powers of the {Maa\ss} lowering operator to our Fourier coefficients of $\mathcal{G}_{-k,D}$ does not yield an expression for which an identity is known that is as nice as the one that Mono uses. We advise the more eager reader to study \cite{mo21}.
\end{rmk}

\end{document}